\documentclass{amsart}

\usepackage{amssymb}
\usepackage[utf8]{inputenc}
\usepackage{tikz}
\usepackage{tikz-cd}
\usepackage{amsmath, amsthm, amssymb, amsfonts}
\usepackage{hyperref}
\usepackage{mathtools}
\usepackage[retainorgcmds]{IEEEtrantools}
\usepackage{latexsym}
\usepackage{enumerate}
\usepackage{nomencl}
\usepackage{nameref}
\usepackage{chngcntr}

\newtheorem{thm}{Theorem}[section]
\newtheorem{lem}[thm]{Lemma}
\newtheorem{prop}[thm]{Proposition}
\newtheorem{cor}[thm]{Corollary}

\theoremstyle{definition}
\newtheorem{defi}[thm]{Definition}
\newtheorem{ex}[thm]{Example}

\theoremstyle{remark}
\newtheorem{rem}[thm]{Remark}

\newcommand{\op}{\normalfont{^{op}}}

\newcommand{\Mod}[1]{\operatorname{mod} #1}
\newcommand{\ind}[1]{\operatorname{ind} #1}
\newcommand{\MOD}[1]{\operatorname{Mod} #1}

\newcommand{\Hom}[3]{\operatorname{Hom}_{#1}(#2,#3) }

\newcommand{\End}[2]{\operatorname{End}_{#1}(#2) }

\newcommand{\SOC}[2]{\operatorname{Soc}_{#1} #2}

\newcommand{\Rad}[1]{\operatorname{Rad} #1}
\newcommand{\RAD}[2]{\operatorname{Rad}^{#1} #2}
\newcommand{\Ima}[1]{\operatorname{Im}#1}
\newcommand{\Ker}[1]{\operatorname{Ker}#1}

\newcommand{\KER}[1]{\operatorname{ker}#1}

\newcommand{\COIM}[1]{\operatorname{coim}#1}
\newcommand{\IMA}[1]{\operatorname{im}#1}
\newcommand{\Tr}[2]{\operatorname{Tr}\left(#1,#2\right)}
\newcommand{\Rej}[2]{\operatorname{Rej}\left(#1,#2\right)}

\newcommand{\N}{\mathbb{N}}

\newcommand{\Proj}[1]{\operatorname{proj.dim}#1}
\newcommand{\Gl}[1]{\operatorname{gl.dim}#1}
\newcommand{\rep}[1]{\operatorname{rep.dim}#1}

\newcommand{\Add}[1]{\operatorname{add}#1}

\newcommand{\LL}[1]{\operatorname{LL}(#1)}

\begin{document}

\title[GR measure, representation dimension and rejective chains]{Gabriel--Roiter measure, representation dimension and rejective chains}
\author{Teresa Conde}
\address{Institute of Algebra and Number Theory, University of Stuttgart\\ Pfaffenwaldring 57, 70569 Stuttgart, Germany}
\email{\href{mailto:tconde@mathematik.uni-stuttgart.de}{\nolinkurl{tconde@mathematik.uni-stuttgart.de}}}
\thanks{I would like to thank Steffen Koenig for the useful discussions and suggestions during the preparation of this manuscript.}
\date{\today}

\begin{abstract}
The Gabriel--Roiter measure is used to give an alternative proof of the finiteness of the representation dimension for Artin algebras, a result established by Iyama in 2002. The concept of Gabriel--Roiter measure can be extended to abelian length categories and every such category has multiple Gabriel--Roiter measures. Using this notion, we prove the following broader statement: given any object $X$ and any Gabriel--Roiter measure $\mu$ in an abelian length category $\mathcal{A}$, there exists an object $X'$ which depends on $X$ and $\mu$, such that $\Gamma=\End{\mathcal{A}}{X\oplus X'}$ has finite global dimension. Analogously to Iyama's original results, our construction yields quasihereditary rings and fits into the theory of rejective chains.
 \end{abstract}

\maketitle
\section{Introduction}
The first Brauer--Thrall conjecture, proved by Roiter in 1968 (\cite{MR0238893}), asserts that every Artin algebra of bounded representation type has finite representation type. The induction scheme used in Roiter's proof prompted Gabriel to introduce a powerful invariant, known as the Gabriel–Roiter measure. The usefulness of the Gabriel--Roiter measure is not limited to the first Brauer--Thrall conjecture. In \cite{grmfringel}, Ringel has used the Gabriel--Roiter measure to give alternative proofs of a number of important results in representation theory published in the 1970's. Ringel employed the Gabriel--Roiter measure to show that every module over an Artin algebra of finite representation type is a direct sum of finitely generated modules, a result originally established in \cite{MR0379578,MR0349740,MR0349747}. In \cite{MR0424874}, Auslander proved that every Artin algebra of infinite representation type possesses an infinitely generated indecomposable module, and it turns out that this result can also be deduced using the Gabriel--Roiter measure. Furthermore, using the notion of Gabriel--Roiter inclusion, Ringel observed that every finitely generated nonsimple indecomposable module $X$ is the extension of an indecomposable module $Y$ having Gabriel--Roiter measure strictly smaller than that of $X$ by some other indecomposable module $Z$. In this paper, we provide yet another application of the Gabriel--Roiter measure. Namely, we use it to prove the finiteness of the representation dimension for Artin algebras, a result first shown by Iyama in 2002.

The representation dimension of an Artin algebra $\Lambda$ is the smallest possible global dimension of an endomorphism algebra $\End{\Lambda}{\overline{X}}$, such that the $\Lambda$-module $\overline{X}$ generates and cogenerates every $\Lambda$-module. This homological dimension was introduced by Auslander in \cite{quenmary}. He proved that the Artin algebras of finite representation type are exactly those whose representation dimension is at most 2 (\cite{quenmary}), hence representation dimension distinguishes finite from infinite representation type. All the classes of algebras for which Auslander determined the precise representation dimension turned out to have representation dimension at most 3. Thus, he asked whether the representation dimension can be greater than 3, but also, whether it always has to be finite. Although an Artin algebra may have arbitrarily large representation dimension (as shown by Rouquier in \cite{MR2231960}), this is always finite. Indeed, Iyama established in \cite{iyama} the following much stronger result. 
\begin{thm}[Iyama]
\label{thm:iyama}
Every finitely generated module $X$ over an Artin algebra $\Lambda$ has a complement $X'$ such that $\Gamma=\End{\Lambda}{X \oplus X'}$ is a quasihereditary algebra.
\end{thm}
The algebra $\Gamma$ in the statement of the theorem must have finite global dimension, as this is the case for all quasihereditary algebras (\cite{Moosonee}). By taking $X$ to be a generator-cogenerator of the module category, it follows as a corollary that all Artin algebras have finite representation dimension. 

Our main goal is to give a new proof of Theorem \ref{thm:iyama} by using the Gabriel--Roiter measure to determine an adequate complement $X'$ of $X$. Using the notion of ``weighted Gabriel--Roiter measure'' (see \cite{MR2346934} and \cite[Appendix, §3]{grmfringel}), one is actually able to constructively define multiple complements $X'$ of a fixed module $X$ such that the resulting Artin algebra $\Gamma$ is quasihereditary. Although our formulae for the complement $X'$ differ from that of Iyama and produce, in general, different quasihereditary algebras $\Gamma$, both constructions share many similarities. Akin to Iyama's original construction, our construction is also associated to rejective chains of subcategories, a notion developed in \cite{2003math.....11281I}. Thus, our setup too produces quasihereditary algebras which satisfy especially nice properties. To be precise, we obtain left strongly quasihereditary algebras, as defined in \cite{RingelIyama}. 

The notion of ``weighted Gabriel--Roiter measure'' is valid for any abelian length category and every such category possesses various Gabriel--Roiter measures. In its most general form, our main result reads as follows.
\begin{thm}[Part of Corollary \ref{cor:finrepdim}]
Let $\mathcal{A}$ be an abelian length category. Given any object $X$ and any Gabriel--Roiter measure $\mu$ in $\mathcal{A}$, there exists an object $X'$ which depends on $X$ and $\mu$, such that $\Gamma=\End{\mathcal{A}}{X\oplus X'}$ is left strongly quasihereditary. As a result, $\Gamma$ has finite global dimension.
\end{thm}

The complete statement of Corollary \ref{cor:finrepdim} provides, in particular, a recipe for constructing distinct generators-cogenerators $\overline{X}$ of $\Mod{\Lambda}$ whose endomorphism algebra $\Gamma=\End{\Lambda}{\overline{X}}$ has finite global dimension and gives an upper bound for the global dimension of the resulting algebra $\Gamma$. For this reason, our main result is potentially useful to compute not only the representation dimension, but also the rigidity dimension of Artin algebras, a notion recently introduced in \cite{2017arXiv170608301C}. 

The layout of this paper is the following. Section \ref{sec:first} contains a direct proof of Theorem \ref{thm:iyama} using the Gabriel--Roiter measure. In Section \ref{sec:second}, we give a succinct account of Iyama's rejective chain theory and its relation to quasihereditary algebras. Finally, in Section \ref{sec:third}, we explain how the proof in Section \ref{sec:first} fits into the general setup described in Section \ref{sec:second}. In this last section, we use the broader notion of Gabriel--Roiter measure studied in \cite{MR2346934} to prove Corollary \ref{cor:finrepdim}, a refined version of Theorem \ref{thm:iyama} for abelian length categories. In particular, we describe how to obtain different complements $X'$ of a fixed finitely generated module $X$ so that $\Gamma=\End{\Lambda}{X \oplus X'}$ has finite global dimension, and provide upper bounds for the projective dimension of the simple $\Gamma$-modules and the global dimension of $\Gamma$.
\subsection*{Notation and conventions}
Throughout this paper, we deal with two types of rings: semiprimary rings and Artin algebras. A unital ring $\Gamma$ is \emph{semiprimary} if its Jacobson radical $\Rad{\Gamma}$ is a nilpotent ideal and $\Gamma / \Rad{\Gamma}$ is semisimple artinian (see \cite[§28]{MR1245487}). Every semiprimary ring $\Gamma$ has finitely many simple modules up to isomorphism, and every $\Gamma$-module has finite Loewy length. However, a semiprimary ring need not have finite length as a module over itself. A unital algebra over a commutative unital artinian ring $k$ is called an \emph{Artin algebra} if it is finitely generated as a $k$-module (see \cite{MR1476671} for details). Every Artin algebra is a semiprimary ring. 

All the modules considered will be left modules. The letter $\Lambda$ shall always denote an Artin algebra over $k$ and $D$ represents the standard duality for Artin algebras. The notation $\MOD{\Lambda}$ will be used for the category of $\Lambda$-modules and $\Mod{\Lambda}$ shall denote the category of finitely generated $\Lambda$-modules. 

All the categories are assumed to be locally small. Given an additive category $\mathcal{A}$, let $\ind{\mathcal{A}}$ be the class containing one representative from each isomorphism class of indecomposable objects in $\mathcal{A}$. Whenever convenient, assume that the elements of $\ind{\mathcal{A}}$ are isomorphism classes of indecomposable objects. For a class $\Theta$ of objects in $\mathcal{A}$, $\Add{\Theta}$ is the full subcategory of $\mathcal{A}$ whose objects are summands of finitary direct sums (i.e.~biproducts) of objects in $\Theta$. If $\Theta=\{X\}$, we write $\Add{X}$ instead of $\Add{\{X\}}$. The notation $\mathbb{N}$ is used for the set natural numbers starting from 1 and $\mathbb{N}_0$ denotes $\mathbb{N} \cup \{0\}$. Lastly, the symbol $\subset$ is used for strict inclusions.
\section{Finiteness of the representation dimension using the Gabriel--Roiter measure}
\label{sec:first}
Within this section, the category $\Mod{\Lambda}$ of finitely generated modules over the Artin algebra $\Lambda$ is simply denoted by $\mathcal{A}$. Note that $\ind \mathcal{A}$ is a set in this case. Given $X$ in $\mathcal{A}$, $|X|$ shall denote the composition length of $X$. 

The Gabriel--Roiter measure of a module $X$ in $\MOD{\Lambda}$ is a set of natural numbers associated to $X$ and its definition relies on a particular total order $\leq$ defined on the power set $\mathcal{P}(\mathbb{N})$ of $\mathbb{N}$. Given two distinct subsets $I$ and $J$ of $\mathbb{N}$, we write $I <J$ if $\min ((I\setminus J)\cup (J\setminus I)) \in J$. This relation defines a \emph{complete} total order on $\mathcal{P}(\mathbb{N})$, meaning that every nonempty subset of $\mathcal{P}(\mathbb{N})$ has a supremum (see \cite[Section 1]{RINGEL2005726}).

\begin{defi}[\cite{MR0340377,RINGEL2005726}]
The \emph{Gabriel--Roiter measure} $\mu (X)$ of a module $X$ in $\MOD{\Lambda}$ is the supremum (taken over $(\mathcal{P}(\mathbb{N}), \leq)$) of the set whose elements are the sets $\{|Y_1|,\ldots,|Y_t|\}$ for which $Y_1 \subset \cdots \subset Y_t$ is a chain of finitely generated indecomposable submodules of $X$.
\end{defi}

It is clear from the definition that the Gabriel--Roiter measure is invariant under isomorphisms. Lemma \ref{lem:propgrm} summarises some of the main properties of the Gabriel--Roiter measure. We refer to \cite[Section 1]{RINGEL2005726} for further details.
\begin{lem}
\label{lem:propgrm}
Let $X$, $Y$ and $Y_1,\ldots,Y_t$ be in $\MOD{\Lambda}$. The following hold:
\begin{enumerate}
\item if $X$ is a submodule of $Y$, then $\mu(X) \leq \mu (Y)$;
\item if $X$ and $Y$ lie in $\ind \mathcal{A}$ and $\mu(X)=\mu(Y)$, then $|X|=|Y|$;
\item $\mu(\bigoplus_{i=1}^t Y_i)=\max \{\mu (Y_i)\mid i =1, \ldots ,t\}$.
\end{enumerate}
\end{lem}
\begin{rem}
We will deal with the Gabriel--Roiter measure of finitely generated modules. However, this invariant is equally well suited to study infinitely generated modules, as observed in \cite{RINGEL2005726,grmfringel}.
\end{rem}

Given $I \subseteq \N$, let $\mathcal{A}(< I)$ be the set of all $Y$ in $\ind{\mathcal{A}}$ whose Gabriel--Roiter measure $\mu(Y)$ is strictly smaller than $I$. According to parts (1) and (3) of Lemma \ref{lem:propgrm}, the set $\mathcal{A}(< I)$ is \emph{closed under cogeneration}, i.e.~the category $\Add{\mathcal{A}(< I)}$ is closed under taking submodules.

For our proof of Theorem \ref{thm:iyama}, we need the notion of trace and reject. Let $X$ be in $\mathcal{A}$ and consider a set $\Theta$ of modules in $\MOD{\Lambda}$. The \emph{reject} of $X$ in $\Theta$ is given by
\[\Rej{X}{\Theta}=\bigcap_{\substack{f \in \Hom{\Lambda}{X}{T},\\ T\in \Theta}} \Ker{f}.\]
The factor module $X/\Rej{X}{\Theta}$ is cogenerated by $\Theta$ and it is indeed the largest quotient of $X$ satisfying such property. The \emph{trace} of $\Theta$ in $X$ is defined by
\[\Tr{\Theta}{X}=\sum_{\substack{f \in \Hom{\Lambda}{T}{X},\\ T\in \Theta}} \Ima{f}.\]
This is the largest submodule of $X$ generated by $\Theta$. Both $\Rej{-}{\Theta}$ and $\Tr{\Theta}{-}$ are subfunctors of the identity functor on $\mathcal{A}$, so they are additive functors and therefore preserve finitary direct sums. We will often consider the situation $\Theta=\{T\}$. In this case, we simply write $\Rej{X}{T}$ and $\Tr{T}{X}$ instead of $\Rej{X}{\{T\}}$ and $\Tr{\{T\}}{X}$.

Given $X \in \mathcal{A}$, we construct recursively a sequence $(X_i)_{i\in \N}$ of quotients of $X$ by setting $X_1=X$ and defining
\[
X_{i+1}=X_i/\Rej{X_i}{\mathcal{A}(<  \mu(X_i))}.
\] 

\begin{lem}
\label{lem:obvi}
If $X_i\neq 0$, then $\mu(X_{i+1})<\mu(X_i)$ and $X_{i+1}$ is a proper quotient of $X_i$. In particular, there exists $\ell \in \N_0$ such that $X_i=0$ for all $i > \ell$.
\end{lem}
\begin{proof}
By definition, $X_{i+1}$ is the largest quotient of $X_i$ which is cogenerated by $\mathcal{A}(< \mu (X_i))$. Using parts (1) and (3) of Lemma \ref{lem:propgrm}, we conclude that $\mu(X_{i+1})< \mu(X_i)$. Hence $X_{i+1}$ is a proper factor of $X_i$. Since $X_1$ is finitely generated, the sequence $(X_i)_{i\in \N}$ must eventually stabilise in $0$. 
\end{proof}

Let $\overline{X}=\bigoplus_{i\in \N} X_i$ -- by Lemma \ref{lem:obvi} this module lies in $\mathcal{A}$. Consider the Artin algebra
\[\Gamma=\End{\Lambda}{\overline{X}}.\]
We prove, in a few steps, that $\Gamma$ is quasihereditary. In order to do so, we must endow the set of simple $\Gamma$-modules with a suitable poset structure.

Let $\Phi$ be the set containing one representative from each isomorphism class of indecomposable summands of $\overline{X}$. The projective indecomposable $\Gamma$-modules are given by $P_{Y}=\Hom{\Lambda}{Y}{\overline{X}}$ for $Y \in \Phi$, hence $\Phi$ labels the simple $\Gamma$-modules. We denote the top of $P_{Y}$ by $L_{Y}$ and endow $\Phi$ with a partial order $\unlhd$ by setting $Y \lhd Z$ if $\mu(Y)>\mu(Z)$. Note that $Y,Z \in \Phi$ are not comparable if $\mu(Y)=\mu(Z)$ but $Y\neq Z$ (i.e.~$Y\not\cong Z$).

The definition of quasihereditary (Artin) algebra is valid, more generally, for any semiprimary ring. Within this section, we only deal with Artin algebras. However, in Sections \ref{sec:second} and \ref{sec:third}, we work with semiprimary quasihereditary rings. For this reason, Definition \ref{defi:qh} is stated in terms of semiprimary rings.
\begin{defi}[\cite{MR1211481,MR987824,RingelIyama}]
\label{defi:qh}
Let $(\Phi,\unlhd)$ be a poset labelling the simple modules over some semiprimary ring $\Gamma$ and denote by $P_Y$ the projective indecomposable module whose (simple) top $L_Y$ corresponds to the label $Y\in \Phi$.
\begin{enumerate}
\item The \emph{standard module} $\Delta_Y$, $Y \in \Phi$, is the largest quotient of $P_Y$ satisfying $\Hom{\Lambda}{P_Z}{\Delta_Y}\neq 0$ only when $Z\unlhd Y$. In other words, $\Delta_Y=P_Y/U_Y$, where $U_Y=\Tr{\{P_Z \mid  Z \ntrianglelefteq Y\}}{P_Y}$.
\item The ring $\Gamma$ is \emph{quasihereditary} with respect to $(\Phi, \unlhd)$ if the following hold for every $Y \in \Phi$:
\begin{enumerate}
\item $U_Y$ has a filtration whose subquotients are of the form $\Delta_Z$, with $Z\rhd Y$;
\item every nonzero morphism $P_Y\longrightarrow \Delta_Y$ is epic.
\end{enumerate}
\item A quasihereditary ring $(\Gamma,\Phi,\unlhd )$ is called \emph{left strongly quasihereditary} if $\Proj{\Delta_Y}\leq 1$ for all $Y\in \Phi$.
\end{enumerate} 
\end{defi}
\begin{rem}
For an Artin algebra, the notion dual to standard module is that of a costandard module. In this case, $\Delta_Y$ is the largest factor module of $P_Y$ whose composition factors are of the form $L_Z$ with $Z \unlhd Y$. The costandard module $\nabla_Y$ is the largest submodule of the injective module with socle $L_Y$ with composition factors of the form $L_Z$, where $Z \unlhd Y$. A quasihereditary algebra is left strongly quasihereditary precisely when the category $\mathcal{F}(\nabla)$ of modules in $\mathcal{A}$ filtered by costandard modules is closed under taking quotients. Different characterisations of left strongly quasihereditary algebras may be found in \cite{yey}, \cite[Lemma 4.1]{MR1211481} and \cite[Appendix]{RingelIyama}.
\end{rem}

We shall see that the Artin algebra $\Gamma=\End{\Lambda}{\overline{X}}$ is left strongly quasihereditary with respect to the poset $(\Phi,\unlhd)$ defined previously. For the next proposition we use the notation in Definition \ref{defi:qh}.
\begin{prop}
\label{prop:auxi}
Let $Y \in \Phi$ and consider the $\Lambda$-module $V_Y=\Rej{Y}{\mathcal{A}(<  \mu(Y))}$, as well as the canonical epic $\pi: Y \longrightarrow Y/V_Y$. Set $\pi^*=\Hom{\Lambda}{\pi}{\overline{X}}$. Then, the $\Gamma$-modules $\Ima{\pi^*}$ and $U_Y$ coincide. Moreover, $U_Y$ is a direct sum of projectives of the form $P_Z$ with $Z\rhd Y$ and $\Proj{\Delta_{Y}} \leq 1$.
\end{prop}
\begin{proof}
By construction of the sequence $(X_i)_{i\in \mathbb{N}}$ and the properties of the rejects, it follows that $Y/V_Y$ lies in $\Add{\overline{X}}$. In fact, $Y/V_Y$ must be isomorphic to a direct sum of modules $Z\in \Phi$ satisfying $\mu(Z)<\mu (Y)$. Hence $\Hom{\Lambda}{Y/V_Y}{\overline{X}}$ is the direct sum of projectives $P_Z$ with $Z\rhd Y$. The functor $\Hom{\Lambda}{-}{\overline{X}}$ maps the epic $\pi$ to the monic $\pi^*:\Hom{\Lambda}{Y/V_Y}{\overline{X}} \longrightarrow P_{Y}$. Using the definition of trace, one deduces that $\Ima{\pi^*}$ is contained in $U_Y$. 

Note that the statement of the proposition will now follow once we show that $U_Y$ is contained in $\Ima{\pi^*}$. For this, consider an arbitrary morphism $f^*:P_{Z} \longrightarrow P_{Y}$ with $Z\ntrianglelefteq Y$. There exists a morphism $f \in \Hom{\Lambda}{Y}{Z}$ such that $\Hom{\Lambda}{f}{\overline{X}}=f^*$. 

Since $Z\ntrianglelefteq Y$, then either $Z\rhd Y$ or $Z$ and $Y$ not comparable. This means that either $\mu(Z)<\mu(Y)$ or $\mu(Z)=\mu(Y)$ but $Z\neq Y$. In both cases we show that $\mu(\Ima{f})<\mu(Y)$. The inclusion $\Ima{f} \subseteq Z$ implies that $\mu(\Ima{f})\leq \mu(Z)$ (here we use part (1) of Lemma \ref{lem:propgrm}). If $\mu(Z)<\mu(Y)$, then $\mu(\Ima{f})\leq \mu(Z)<\mu(Y)$. In case $\mu(Z)=\mu(Y)$ and $Z\neq Y$, the modules $Y$ and $Z$ are not isomorphic but they have the same length by part (2) of Lemma \ref{lem:propgrm}. Hence, the map $f$ cannot be an epic in this case. If $\Ima{f}$ and $Z$ had the same Gabriel--Roiter measure, then parts (2) and (3) of Lemma \ref{lem:propgrm} would imply the existence an indecomposable summand of $\Ima{f}$ whose inclusion in $Z$ was an isomorphism. The latter situation only occurs when $f$ is an epic, which shows that $\mu(\Ima{f})<\mu(Z) =\mu(Y)$.

By construction, $Y/V_Y$ is the largest factor of $Y$ whose Gabriel--Roiter measure is smaller than $\mu(Y)$. Since $\Ima{f}$ is a quotient of $Y$ satisfying $\mu(\Ima{f})<\mu(Y)$, then the epic $\COIM{f}:Y \longrightarrow \Ima{f}$ factors through $\pi: Y\longrightarrow Y/V_Y$. The functor $\Hom{\Lambda}{-}{\overline{X}}$ maps
\[
\begin{tikzcd}[ampersand replacement=\&]
Y \arrow[two heads]{rd}{\COIM{f}} \arrow[two heads]{d}{\pi}\arrow{rr}{f} \&\& Z \\
Y/V_Y \arrow[two heads]{r}{\rho} \& \Ima{f} \arrow[hook]{ru}[swap]{\IMA{f}}\&
\end{tikzcd}
\]
to
\[
\begin{tikzcd}[ampersand replacement=\&]
P_{Z}  \arrow{rd}[swap]{(\IMA{f})^*} \arrow{rr}{f^*} \&\& P_{Y} \\
\& \Hom{\Lambda}{\Ima{f}}{\overline{X}} \arrow[hook]{r}{\rho^*} \arrow[hook]{ru}{(\COIM{f})^*} \& \Hom{\Lambda}{Y/V_Y}{\overline{X}} \arrow[hook]{u}{\pi^*}
\end{tikzcd},
\]
where $g^*$ denotes $\Hom{\Lambda}{g}{\overline{X}}$ for a map $g$ in $\mathcal{A}$. From this diagram it is clear that $\Ima{f^*}\subseteq \Ima{\pi^*} $. Since the morphism $f^*$ was chosen arbitrarily, we conclude that $U_Y \subseteq \Ima{\pi^*}$.
\end{proof}

We now use Proposition \ref{prop:auxi} to prove that $\Gamma$ is quasihereditary.
\begin{thm}
\label{thm:iyamat}
The Artin algebra $\Gamma$ is quasihereditary, hence $\Gl{\Gamma}<\infty$. More precisely, $\Gamma$ is left strongly quasihereditary with respect to $(\Phi, \unlhd)$.
\end{thm}
\begin{proof}
We start by checking condition (2)(b) in Definition \ref{defi:qh}. This is equivalent to showing that every nonisomphism $f^*:P_{Y} \longrightarrow P_{Y}$ satisfies $\Ima{f^*}\subseteq  U_Y$. We follow the strategy used in part of the proof of Proposition \ref{prop:auxi}. Let $f^*:P_{Y} \longrightarrow P_{Y}$ be a nonisomorphism. There exists a morphism $f \in \End{\Lambda}{Y}$ such that $\Hom{\Lambda}{f}{\overline{X}}=f^*$. Using an argument identical to the one in the proof of Proposition \ref{prop:auxi}, we must have $\mu(\Ima{f})<\mu (Y)$, otherwise $f$, and consequently $f^*$, would be isomorphisms. But then, as before, $f$ factors through the epic $\pi:Y\longrightarrow Y/V_Y$, thus $\Ima{f^*} \subseteq \Ima{\pi^*}=U_Y$ (here the last equality follows from Proposition \ref{prop:auxi}).

It remains to check condition (a) in the definition of quasihereditary algebra. We do this by descending induction on the poset $(\Phi,\unlhd)$. If $Y$ is a maximal element in $(\Phi,\unlhd)$, then $U_Y=0$ according to Proposition \ref{prop:auxi}, so the condition is trivially satisfied. Suppose now that $Y$ is not maximal and assume that $U_W$ is filtered by modules $\Delta_{Z}$ with $Z \rhd W$, for every $W$ satisfying $W\rhd Y$. By Proposition \ref{prop:auxi}, $U_Y$ is a direct sum of modules $P_W$ with $W \rhd Y$, so this module must be filtered by $\Delta_{Z}$, with $Z \rhd Y$. This shows that $\Gamma$ is quasihereditary. It follows from Proposition \ref{prop:auxi} that $\Gamma$ is left strongly quasihereditary. By \cite[Theorem $4.3$]{Moosonee}, $\Gl{\Gamma}< \infty$.
\end{proof}

\begin{rem}
In Section \ref{sec:third}, we generalise Theorem \ref{thm:iyamat} and provide upper bounds for the global dimension of $\Gamma$ and for the projective dimension of the simple $\Gamma$-modules.
\end{rem}

\begin{ex}
\label{ex:easy}
Let $\Lambda$ be the bound quiver algebra $k Q/I$ where $I=\langle \alpha^3, \beta\alpha\rangle$ and
\[Q=
\begin{tikzcd}[ampersand replacement=\&]
  \overset{1}{\circ} \arrow[loop left]{}{\alpha} \arrow{r}{\beta}\&
    \overset{2}{\circ}
\end{tikzcd}.
\]
Let $X_1=X$ be the projective indecomposable $P_1$ -- note that $\Gl{\End{\Lambda}{X}}=\infty$. The Gabriel--Roiter measure of $X_1$ is $\mu(X_1)=\{1,2,4\}$ and
\[
X_2=X_1/\Rej{X_1}{\mathcal{A}(<  \mu (X_1))}=\begin{tikzcd}[ampersand replacement=\&, row sep=tiny, column sep=tiny,]
\& 1 \arrow[dash]{dr} \arrow[dash]{dl} \& \\
1  \& \& 2 
\end{tikzcd}.
\]
Now $\mu (X_2)=\{1,3\}$, $X_3=X_2/\Rej{X_2}{\mathcal{A}(<  \mu (X_2))}=L_1$ and $X_4=0$. Hence the endomorphism algebra $\Gamma=\End{\Lambda}{X_1 \oplus X_2 \oplus X_3}$ has finite global dimension.
\end{ex}

Recall that the \textit{representation dimension} of $\Lambda$ is given by
\[
\rep{\Lambda}=\min\{\Gl{\End{\Lambda}{\overline{X}}}\mid \overline{X}\text{ generates and cogenerates }\Mod{\Lambda}\}.
\]
\begin{cor}[{\cite[Corollary $1.2$]{iyama}}]
Every finitely generated module $X$ over an Artin algebra $\Lambda$ has a complement $X'$, such that $\Gl{\End{\Lambda}{X\oplus X'}}< \infty$. In particular, every Artin algebra has finite representation dimension.
\end{cor}

\begin{proof}
The first statement of the corollary is a consequence of Theorem \ref{thm:iyamat}. The second statement follows from the first by setting $X=\Lambda\oplus D(\Lambda)$.\end{proof}

\section{Rejective chains and quasihereditary algebras}
\label{sec:second}
The theory of rejective chains was developed by Iyama, in connection with his proofs of the finiteness of representation dimension for Artin algebras and of Solomon's second conjecture on local zeta functions for orders (\cite{iyama,MR1953711}). In \cite{2003math.....11281I}, Iyama showed that every left (resp.~right) complete total rejective chain of finite type gives rise to a left (resp.~right) strongly quasihereditary ring. It turns out that the notion of complete total rejective chain is, in a specific sense, equivalent to that of a strongly quasihereditary ring, as observed in \cite[Theorem 3.22]{Mayu}. In particular, the quasihereditary algebra defined in Section \ref{sec:first} is associated to a rejective chain (this will be clarified in Section \ref{sec:third}). In this section we give a concise self-contained account of the basics on rejective chains.

From now onwards, $\mathcal{A}$ will be a fixed additive category. All subcategories are assumed to be full and closed under isomorphisms, finitary direct sums and direct summands.

Let $\mathcal{C}'$ be a subcategory of some category $\mathcal{C}$. Recall that a morphism $f$ in $\Hom{\mathcal{C}}{X}{X'}$ is called a \emph{left $\mathcal{C}'$-approximation} of an object $X$ in $\mathcal{C}$ if $X'$ lies in $\mathcal{C}'$ and every morphism $f'\in \Hom{\mathcal{C}}{X}{X''}$ with $X''$ in $\mathcal{C}'$ factors through $f$. The notion of right $\mathcal{C}'$-approximation of an object in $\mathcal{C}$ is defined dually (see \cite[§1]{ausrei}). In addition, the category $\mathcal{C'}$ is called a \emph{reflective} (resp.~\emph{coreflective}) subcategory of $\mathcal{C}$, if the inclusion functor has a left (resp.~right) adjoint.

\begin{defi}[{\cite[$\S 1.5$]{2003math.....11281I}}]
A subcategory $\mathcal{C}'$ of a category $\mathcal{C}$ is said to be \emph{left} (resp.~\emph{right}) \emph{rejective} if every $X$ in $\mathcal{C}$ has a left (resp.~right) $\mathcal{C}'$-approximation which is an epic (resp.~monic). Equivalently, $\mathcal{C'}$ is a left (resp.~right) rejective subcategory of $\mathcal{C}$, if $\mathcal{C'}$ is a reflective (resp.~coreflective) subcategory of $\mathcal{C}$ and the unit (resp.~counit) of the corresponding adjunction is an epic (resp.~monic) in the category of endofunctors of $\mathcal{C}$.
\end{defi}

Next, we state the definition of (left and right) complete total rejective chain, introduced by Iyama in \cite[$\S 2.2$]{2003math.....11281I}, and the broader notion of complete total prerejective chain.
\begin{defi}
A chain $0=\mathcal{C}_{\ell+1}\subset \mathcal{C}_{\ell}  \subset \cdots \subset \mathcal{C}_{1}$ of proper inclusions of subcategories of $\mathcal{A}$ is called a \emph{left} (resp.~\emph{right}) \emph{complete total prerejective chain} of length $\ell$, if $\mathcal{C}_i$ is a left (resp.~right) rejective subcategory of $\mathcal{C}_1$ and every morphism between nonisomorphic indecomposable objects in $\mathcal{C}_{i}$ factors through an object in $\mathcal{C}_{i+1}$, for every $i$ satisfying $1\leq i \leq \ell$.

The chain is \emph{left} (resp.~\emph{right}) \emph{complete total rejective} if $\mathcal{C}_{i}$ is a left (resp.~right) rejective subcategory of $\mathcal{C}_1$ and every nonisomorphism between indecomposable objects in $\mathcal{C}_i$ factors through an object in $\mathcal{C}_{i+1}$, for every $i$ satisfying $1\leq i \leq \ell $.
\end{defi}
\begin{rem}
Every left (resp.~right) complete total rejective chain is left (resp.~right) complete total prerejective.
\end{rem}

\begin{ex}
Consider the Kronecker algebra, i.e.~the path algebra $\Lambda=k Q$ where
\[Q=
\begin{tikzcd}[ampersand replacement=\&]
  \overset{1}{\circ} \arrow[bend left]{r}  \arrow[bend right]{r}\&
    \overset{2}{\circ}
\end{tikzcd}.
\]
Let $X$ be any regular $\Lambda$-module with $|X|\geq 4$ and let $L_1$ be the simple (injective) module associated to vertex 1 (see, for instance, \cite[Chapter VII.7]{MR1476671} for details about the indecomposables over the Kronecker algebra). Then $0\subset \Add{L_1} \subset \Add{X\oplus L_1}$ is a left complete total prerejective chain in $\Mod{\Lambda}$ which is not left complete total rejective.
\end{ex}

Let $e$ an idempotent endomorphism of an object $X$ in a category $\mathcal{C}$. The morphism $e$ is called a \emph{split idempotent} if there is $Y$ in $\mathcal{C}$, $\pi \in \Hom{\mathcal{A}}{X}{Y}$ and $\iota \in \Hom{\mathcal{C}}{Y}{X}$ satisfying $e=\iota \circ \pi$ and $\pi \circ \iota=1_{Y}$. We say that an additive category $\mathcal{C}$ has \emph{finite type} if the following conditions hold:
\begin{enumerate}
\item $\mathcal{C}$ is \emph{idempotent complete}, i.e. all idempontent endomorphisms in $\mathcal{C}$ are split;
\item $\mathcal{C}$ contains only finitely many indecomposable objects up to isomorphism;
\item the endomorphism ring of any object in $\mathcal{C}$ is semiprimary.
\end{enumerate}
According to \cite[Corollary $4.4$]{MR3431480}, every category of finite type is Krull-Schmidt. Recall that an additive category is \emph{Krull-Schmidt} if every object admits a finite direct sum decomposition into indecomposable objects having local endomorphism rings. By our assumptions on subcategories, an indecomposable object in a subcategory $\mathcal{C'}$ of a category $\mathcal{C}$ is also an indecomposable object in $\mathcal{C}$; similarly, finitary direct sums of objects in $\mathcal{C}'$ are also (finitary) direct sums in $\mathcal{C}$. Additionally, note that if $\mathcal{C'}$ is a subcategory of $\mathcal{C}$ and $\mathcal{C}$ has finite type, then so does $\mathcal{C'}$.

Suppose $0=\mathcal{C}_{\ell+1}\subset \mathcal{C}_{\ell}  \subset \cdots \subset \mathcal{C}_{1}$ is a left or right complete total prerejective chain of subcategories of $\mathcal{A}$ and assume that $\mathcal{C}_{1}$ has finite type. Then $\mathcal{C}_1$ is the additive closure of some object $\overline{X}$ which is the direct sum of a finite number of indecomposable objects and $\Gamma=\End{\mathcal{A}}{\overline{X}}$ is a semiprimary ring. For every indecomposable object $Y$ in $\mathcal{C}_1=\Add{\overline{X}}$ there exists a unique $i$ satisfying $1\leq i \leq\ell$, such that $Y$ lies in $\mathcal{C}_i$ but not in $\mathcal{C}_{i+1}$ -- we denote such $i$ by $\ell_Y$ and call it the \emph{level} of $Y$. The level of an indecomposable object only depends on its isomorphism class, since, by assumption, our subcategories are closed under isomorphisms. Let $\Phi$ be a set containing one representative from each isomorphism class of indecomposable objects in $\mathcal{C}_{1}$. We define a partial order $\unlhd$ in $\Phi$ by setting $Y\lhd Z$ if $\ell_Y < \ell_{Z}$. The projective indecomposable $\Gamma$-modules are given by $P_Y=\Hom{\mathcal{A}}{Y}{\overline{X}}$ with $Y\in \Phi$, so the poset $(\Phi, \unlhd)$ labels the simple $\Gamma$-modules. Analogously, the $\Gamma\op$-modules $P_Y\op=\Hom{\mathcal{A}}{\overline{X}}{Y}$ with $Y \in \Phi$ form a complete set of projective indecomposable $\Gamma\op$-modules, up to isomorphism.

For the next result, recall the notion of trace in Section \ref{sec:first} and the notation introduced in Definition \ref{defi:qh}. 

\begin{lem}
\label{lem:prerej}
Suppose $0=\mathcal{C}_{\ell +1}\subset \mathcal{C}_{\ell}  \subset \cdots \subset \mathcal{C}_{1}$ is a left complete total prerejective chain of subcategories of $\mathcal{A}$. Assume that $\mathcal{C}_1$ has finite type, let $\overline{X}$ be an additive generator of $\mathcal{C}_1$ and define $\Gamma=\End{\mathcal{A}}{\overline{X}}$. Take $Y\in \Phi$, let $\pi:Y \longrightarrow Y'$ be any left $\mathcal{C}_{\ell_{Y}+1}$-approximation of $Y$ and set $\pi^*=\Hom{\mathcal{A}}{\pi}{\overline{X}}$. Then the $\Gamma$-modules $\Ima{\pi^*}$ and $U_Y$ coincide. In fact, $U_Y$ a direct sum of projectives $P_Z$ with $Z\rhd Y$.
\end{lem}
\begin{proof}
Let $Y$ be in $\Phi$. Then $Y$ lies in $\mathcal{C}_1$. Since $\mathcal{C}_{\ell_Y+1}$ is a left rejective subcategory of $\mathcal{C}_1$, there exists a left $\mathcal{C}_{\ell_Y+1}$-approximation of $Y$ which is an epic. Let $\pi:Y \longrightarrow Y'$ be such an approximation ($\pi$ is actually a left minimal $\mathcal{C}_{\ell_Y+1}$-approximation). Any other left $\mathcal{C}_{\ell_Y+1}$-approximation $\nu$ of $Y$ satisfies $\Ima{\pi^*}=\Ima{\nu^*}$, where $g^*$ denotes $\Hom{\mathcal{A}}{g}{\overline{X}}$ for a morphism $g$ in $\mathcal{A}$. The functor $\Hom{\mathcal{A}}{-}{\overline{X}}$ maps the epic $\pi$ to the monic
\[\pi^*: \Hom{\mathcal{A}}{Y'}{\overline{X}}\longrightarrow P_Y.\]
Since $Y'$ lies in $\mathcal{C}_{\ell_Y+1}$, then $Y'$ is isomorphic to a finite direct sum of indecomposable objects $Z \in \Phi$, with $\ell_Z >\ell_Y$. The $\Gamma$-module $\Hom{\mathcal{A}}{Y'}{\overline{X}}\cong \Ima{\pi^*}$ is then isomorphic to a finite direct sum of projective $\Gamma$-modules $P_Z$, with $Z \rhd Y$. In particular the monic $\pi^*$ factors through the embedding of $\Gamma$-module $U_Y$ in $P_Y$. Hence $\Ima{\pi^*}\subseteq U_Y$ and it remains to show these two $\Gamma$-modules actually coincide. 

For this, consider an arbitrary morphism in $\Hom{\Gamma}{P_Z}{P_Y}$, where $Z\ntrianglelefteq Y$. Any such map is of the form $f^*$, i.e.~it is the image of a morphism $f:Y\longrightarrow Z$ through the functor $\Hom{\mathcal{A}}{-}{\overline{X}}$. Since $Z\ntrianglelefteq Y$, then either $Z\rhd Y$, or otherwise $Z$ and $Y$ are not comparable in the poset $(\Phi,\unlhd)$. In both cases, we prove that $f$ factors through $\pi$. If $Z\rhd Y$, then $\ell_Z \geq \ell_Y +1$, so $Z$ belongs to $\mathcal{C}_{\ell_Y+1}$ and $f$ must factor through the left $\mathcal{C}_{\ell_Y+1}$-approximation $\pi$. If $Z$ and $Y$ are not comparable, then $\ell_Y=\ell_Z$, but $Z \neq Y$ (i.e.~$Z\not\cong Y$). The definition of complete total left prerejective chain implies then that $f:Y\longrightarrow Z$ factors through an object in $\mathcal{C}_{\ell_Y+1}$, therefore $f$ must also factor through the left $\mathcal{C}_{\ell_Y+1}$-approximation $\pi$. So in both cases, we have $f=h\circ \pi$ for some morphism $h$ in $\mathcal{A}$. Consequently, $f^*$ factors as
\[
\begin{tikzcd}[ampersand replacement=\&]
P_{Z}  \arrow{rd}{h^*} \arrow{rr}{f^*} \&\& P_{Y} \\
\& \Hom{\mathcal{A}}{Y'}{\overline{X}} \arrow[hook]{ru}{\pi^*} \& 
\end{tikzcd}.
\]
It follows that $\Ima{f^*}\subseteq \Ima{\pi^*}$. Since the morphism $f^*:P_Z \longrightarrow P_Y$ was chosen arbitrarily among those satisfying $Z\ntrianglelefteq Y$, we deduce that $U_Y= \Ima{\pi^*}$.
\end{proof}

Not surprisingly, Lemma \ref{lem:prerej} has a corresponding dual version. We state it for completeness.

\begin{lem}
Suppose $0=\mathcal{C}_{\ell +1}\subset \mathcal{C}_{\ell}  \subset \cdots \subset \mathcal{C}_{1}$ is a right complete total prerejective chain of subcategories of $\mathcal{A}$. Assume that $\mathcal{C}_1$ has finite type, let $\overline{X}$ be an additive generator of $\mathcal{C}_1$ and define $\Gamma=\End{\mathcal{A}}{\overline{X}}$. Take $Y\in \Phi$, let $\iota:Y' \longrightarrow Y$ be any right $\mathcal{C}_{\ell_{Y}+1}$-approximation of $Y$ and set $\iota_*=\Hom{\mathcal{A}}{\overline{X}}{\iota}$. Then the $\Gamma\op$-modules $\Ima{\iota_*}$ and $U_Y$ coincide. In fact, $U_Y$ a direct sum of projectives $P_Z\op$ with $Z\rhd Y$.
\end{lem}

As we shall see next, complete total prerejective chains give rise to standardly stratified rings. Note that $\Gamma$ is called \emph{right standardly stratified} with respect to a labelling poset $(\Phi,\unlhd)$ if it satisfies condition (a) in the definition of quasihereditary ring (see Definition \ref{defi:qh}), i.e.~if $U_Y$ is filtered by standard modules $\Delta_Z$ with $Z\rhd Y$ for every $Y \in \Phi$.

\begin{prop}
\label{prop:reject}
Suppose $0=\mathcal{C}_{\ell +1}\subset \mathcal{C}_{\ell}  \subset \cdots \subset \mathcal{C}_{1}$ is a left (resp.~right) complete total prerejective chain of subcategories of $\mathcal{A}$ and assume that $\mathcal{C}_1$ has finite type. Let $\overline{X}$ be an additive generator of $\mathcal{C}_1$ and define $\Gamma=\End{\mathcal{A}}{\overline{X}}$. Then $\Gamma$ (resp.~$\Gamma\op$) is a right standardly stratified ring with respect to $(\Phi, \unlhd)$. Furthermore, every standard $\Gamma$-module (resp.~$\Gamma\op$-module) has projective dimension at most $1$.
\end{prop}

\begin{proof}
We prove the statement of the proposition for left complete total prerejective chains. From Lemma \ref{lem:prerej}, it is clear that every standard $\Gamma$-module has projective dimension at most one. We show, by descending induction on $(\Phi,\unlhd)$, that $U_Y$ is filtered by modules $\Delta_{Z}$ with $Z\rhd Y$, for every $Y\in\Phi$. If $Y$ is a maximal element in $(\Phi,\unlhd)$ then $U_Y=0$ by Lemma \ref{lem:prerej}, so the condition is trivially satisfied. Suppose now that $Y$ is not maximal and assume that $U_W$ is filtered by modules $\Delta_{Z}$ with $Z \rhd W$, for every $W\in \Phi$ satisfying $W\rhd Y$. By Lemma \ref{lem:prerej}, $ U_Y$ is a direct sum of projectives $P_W$, with $W\rhd Y$, so $U_Y$ must be filtered by $\Delta_{Z}$, with $Z \rhd Y$.
\end{proof}

Theorem \ref{thm:reject} is a slight refinement of results proved in \cite[$\S 2.2.2$, $\S 3.5.1$]{2003math.....11281I}.
\begin{thm}
\label{thm:reject}
Suppose $0=\mathcal{C}_{\ell +1}\subset \mathcal{C}_{\ell}  \subset \cdots \subset \mathcal{C}_{1}$ is a left (resp.~right) complete total rejective chain of subcategories of $\mathcal{A}$ and assume that $\mathcal{C}_1$ has finite type. Let $\overline{X}$ be an additive generator of $\mathcal{C}_1$ and define $\Gamma=\End{\mathcal{A}}{\overline{X}}$. Then $\Gamma$ (resp.~$\Gamma\op$) is a left strongly quasihereditary ring with respect to $(\Phi, \unlhd)$ and the projective dimension of the simple module labelled by $Y$ is at most $\ell_Y$ for every $Y\in \Phi$. In particular, $\Gl{\Gamma}\leq \ell$.
\end{thm}
\begin{proof}
We prove the statement for left complete total rejective chains -- the dual case is analogous. Consider $Y \in \Phi$. We start by checking condition (2)(b) in Definition \ref{defi:qh}. This amounts to proving that every nonisomorphism $f^*:P_Y \longrightarrow P_Y$ satisfies $\Ima{f^*}\subseteq U_Y$. As before, there is $f \in \End{\mathcal{A}}{Y}$ such that $f^*=\Hom{\mathcal{A}}{f}{\overline{X}}$. The morphism $f$ cannot be an isomorphism. Hence, by the definition of left complete total rejective chain, $f$ factors through a module in $\mathcal{C}_{\ell_Y+1}$. More concretely, $f$ must factor through a left $\mathcal{C}_{\ell_Y+1}$-approximation $\pi:Y\longrightarrow Y'$ of $Y$. Therefore $\Ima{f^*}\subseteq \Ima{\pi^*}=U_Y$, where the equality follows from Lemma \ref{lem:prerej}. This conclusion (together with Proposition \ref{prop:reject}) implies that $\Gamma$ is left strongly quasihereditary.

It remains to show that $\Proj{L_Y}\leq \ell_Y$. The upper bound for the global dimension of $\Gamma$ will then follow from the fact that $\ell=\max\{\ell_Y\mid Y \in \Phi\}$ and from \cite[Corollary 11]{MR0074406}. So consider $Y \in \Phi$. If $\ell_Y=1$, then $Y$ is a minimal element in the poset $(\Phi,\unlhd)$, hence $\Delta_Y\cong L_Y$ as $\Gamma$ is quasihereditary. Proposition \ref{prop:reject} then implies that $\Proj{L_Y}\leq 1$. Suppose now that $\ell_Y>1$ and assume that $\Proj{L_W}\leq \ell_W$ for every $W \in \Phi$ satisfying $\ell_W <\ell_Y$ (that is, for every $W$ such that $W \lhd Y$). Since $\Rad{\Delta_Y}$ has finite Loewy length and every radical layer of $\Rad{\Delta_Y}$ is a (possibly infinite) direct sum of simples $L_{W}$ with $W \lhd Y$, then $\Proj{\Rad{\Delta_Y}}\leq \ell_Y -1$ by the Horseshoe Lemma. Consider now the following short exact sequence
\[
\begin{tikzcd}[ampersand replacement=\&]
0 \arrow{r} \& U_Y\arrow{r} \& \Rad{P_{Y}} \arrow{r} \& \Rad{\Delta_{Y}} \arrow{r} \& 0
\end{tikzcd}.
\]
By Lemma \ref{lem:prerej}, $U_Y$ is projective. As $\Proj{\Rad{\Delta_{Y}}}\leq \ell_{Y} -1$, the Horseshoe Lemma implies that $\Proj{\Rad{P_{Y}}}\leq \ell_{Y} -1$. Hence $\Proj{L_{Y}} \leq \ell_{Y}$.
\end{proof}

\begin{ex}
\label{ex:adr}
The \emph{ADR algebra} of (an Artin algebra) $\Lambda$ is given by the endomorphism algebra $\End{\Lambda}{\bigoplus_{i=1}^{\LL{\Lambda}}\Lambda/\RAD{i}{\Lambda}}$. Here $\LL{X}$ denotes the Loewy length of a module $X$ in $\Mod{\Lambda}$. The ADR algebra must arise from a rejective chain, as it is a strongly quasihereditary algebra (\cite{MR3510398,Mayu}). We clarify the latter assertion. Consider the subcategories $\mathcal{C}_i=\Add{(\bigoplus_{j=1}^{\LL{\Lambda}-i+1}\Lambda/\RAD{j}{\Lambda})}$ of $\mathcal{A}=\Mod{\Lambda}$ -- these turn out to define a left complete total rejective chain. Using Theorem \ref{thm:reject}, we recover the following result due to Auslander (\cite{MR0349747}): the global dimension of the ADR algebra of $\Lambda$ is bounded above by $\LL{\Lambda}$. More generally, the broader class of ADR algebras considered in \cite{xi} stems from left complete total rejective chains, as explained in \cite[Section 2]{Mayu}.
\end{ex}
\begin{ex}
By dualising the construction in Example \ref{ex:adr}, one deduces that $\mathcal{C}_i=\Add{(\bigoplus_{j=i}^{\LL{D(\Lambda)}}\SOC{j}{D(\Lambda)})}$ defines a right complete total rejective chain of $\mathcal{A}=\Mod{\Lambda}$. 
\end{ex}
\begin{ex}
\label{ex:iyama}
As previously mentioned, the quasihereditary algebras constructed by Iyama for his proof of Theorem \ref{thm:iyama} stem from rejective chains. We outline Iyama's construction and refer to \cite[§2]{iyama}, \cite{2003math.....11281I} or \cite{RingelIyama} for further details. Let $X=X_1$ be in $\Mod{\Lambda}$ and define recursively
\[
X_{i+1}=\sum_{f \in \Rad(X_i,X_i)} \Ima{f},
\]
where $\Rad(X_i,X_i)$ denotes the set of radical morphisms from $X_i$ to itself. The module $X_{i+1}$ is a proper submodule of $X_i$, so this sequence of modules stabilises in $0$. By setting $\mathcal{C}_i=\Add{\bigoplus_{j\geq i} X_j}$, we obtain a right complete total rejective chain.
\end{ex}
\begin{ex}
Let $X=X_1$ be in $\Mod{\Lambda}$ and define recursively
\[
X_{i+1}=X_i/ \bigcap_{f \in \Rad(X_i,X_i)} \Ker{f}.
\]
This construction is obtained by dualising the one in Example \ref{ex:iyama}. In this case, the categories $\mathcal{C}_i=\Add{\bigoplus_{j\geq i} X_j}$ define a left complete total rejective chain.
\end{ex}
\begin{ex}
The quasihereditary algebras associated to reduced words in the Coxeter group of preprojective algebras studied in \cite{IyamaReiten} and \cite{KacMoodyAdv} also stem from rejective chains. The algebras in \cite{IyamaReiten} are associated to left complete total rejective chains. The setup in \cite{KacMoodyAdv} is dual to to the one in \cite{IyamaReiten} and the corresponding algebras come from right complete total rejective chains. 
\end{ex}
\section{Gabriel--Roiter measure and rejective chains}
\label{sec:third}
In Section \ref{sec:first}, we have used the Gabriel--Roiter measure to show that every $X$ in $\mathcal{A}=\Mod{\Lambda}$ has a complement $X'$ such that $\End{\mathcal{A}}{X\oplus X'}$ is a left strongly quasihereditary algebra. Now, we extend our proof to the case of $\mathcal{A}$ being an abelian length category and use a broader notion of Gabriel--Roiter measure. In theory, this will allow us to produce multiple complements $X'$ of a fixed object $X$ so that the resulting ring $\End{\mathcal{A}}{X\oplus X'}$ is left strongly quasihereditary. It is not new that Iyama's result (Theorem \ref{thm:iyama}) holds in abelian length categories. It is also not unexpected that one obtains a general procedure for producing multiple complements $X'$. As a matter of fact, the construction of $X'$ in Section \ref{sec:first} only depends on the properties of the Gabriel--Roiter measure stated in Lemma \ref{lem:propgrm}, therefore it seems reasonable that the method should work for other functions $\mu$ with ``similar features''. Instead of adapting the arguments in Section \ref{sec:first}, we will employ the technology of rejective chains to generalise and refine Theorem \ref{thm:iyamat}.

For this, we fix an abelian length category $\mathcal{A}$. Note that an abelian category is a \emph{length category} if every object satisfies the ascending chain condition and the descending chain condition on subobjects and if the isomorphism classes of indecomposable objects form a set (\cite{MR0340377}). As a consequence of the conditions imposed, $\mathcal{A}$ is a Krull-Schmidt category and every object in $\mathcal{A}$ has a composition series (\cite{MR0340377}, \cite[Section 2]{MR2346934}). According to \cite[Corollary $4.4$]{MR3431480}, the endomorphism ring of any object in $\mathcal{A}$ has semisimple artinian top. Furthermore, the Harada--Sai Lemma for abelian categories implies that the radical of $\End{\mathcal{A}}{X}$ is nilpotent for every $X$ in $\mathcal{A}$ (see \cite{MR0286859} and also \cite[Lemma 8]{MR3373366}). Consequently, the endomorphism ring of any object in $\mathcal{A}$ is semiprimary (consult also \cite[Corollary $23.9$]{MR1245487}). 

\begin{defi}[{\cite[Section 3]{MR2346934}}]
A \emph{length function} on $\mathcal{A}$ is a map $\lambda:\mathcal{A}\longrightarrow [0,+\infty[$ satisfying:
\begin{enumerate}
\item $\lambda (X)=0$ if and only if $X=0$;
\item $\lambda (Y_2)=\lambda (Y_1) + \lambda (Y_3)$ for every short exact sequence
\[\begin{tikzcd}[ampersand replacement=\&]
0 \arrow{r} \& Y_1 \arrow{r} \& Y_2 \arrow{r} \& Y_3 \arrow{r} \& 0.
\end{tikzcd}\]
\end{enumerate}
\end{defi}
\begin{rem}
A length function is totally determined by its values on simple objects. By setting all the simples to have value one, one gets the usual composition length.
\end{rem}
Let $\lambda$ be a length function on $\mathcal{A}$. Following \cite[Section 3]{MR2346934}, we explain how to define a ``weighted Gabriel--Roiter measure'' $\mu_{\lambda}$ of $\mathcal{A}$ associated to $\lambda$. 

Let $\mathcal{P}_f([0,+\infty[)$ be the set whose elements are finite subsets of ${[0,+\infty[}$. This becomes a total order when endowed with the lexicographic order $\leq$, which is defined as $I<J$ if $\min ((I\setminus J)\cup (J\setminus I)) \in J$. Now, observe that $\ind{\mathcal{A}}$ becomes a partially ordered set by setting $Z\subseteq Y$ if $Z$ is a subobject of $Y$. Given $Y$ in $\ind{\mathcal{A}}$, denote by $\mu_{\lambda}(Y)$ the maximum (taken over $(\mathcal{P}_f([0,+\infty[),\leq)$) of the set whose elements are the sets $\{\lambda(Y_1),\ldots \lambda(Y_t)\}$ for which $Y_1 \subset \cdots \subset Y_t$ is a chain of subobjects of $Y$ in $(\ind{\mathcal{A}},\subseteq)$. One may extend the domain of $\mu_{\lambda}$ to $\mathcal{A}$, by defining
\[
\mu_{\lambda}(\bigoplus_{i=1}^t Y_i)=\max \{{\mu}_\lambda(Y_i)\mid i=1,\ldots, t\}.
\]
The function $\mu_{\lambda}$ is called the \emph{Gabriel--Roiter measure of $\mathcal{A}$ associated to $\lambda$}. By specifying $\lambda$ to be the composition length and $\mathcal{A}$ to be $\Mod{\Lambda}$, one obtains the classic definition of Gabriel--Roiter measure for finitely generated modules. The function $\mu_{\lambda}$ satisfies properties analogous to those of the classical Gabriel--Roiter measure.

\begin{thm}[{\cite[Theorem $3.1$, Proposition $3.2$]{MR2346934}}]
\label{thm:easy}
Let $X$, $Y$ be objects in $\mathcal{A}$. Let $\mu_{\lambda}$ be the Gabriel--Roiter measure of $\mathcal{A}$ associated to a length function $\lambda$. Then:
\begin{enumerate}
\item if $X$ is a subobject of $Y$, then $\mu_{\lambda}(X)\leq \mu_{\lambda}(Y)$;
\item if $X$ and $Y$ lie in $\ind{\mathcal{A}}$ and $\mu_{\lambda}(X)=\mu_{\lambda}(Y)$, then $\lambda (X)=\lambda (Y)$.
\end{enumerate}
\end{thm}
Before moving on, we need to extend the notion of reject. For this, suppose that $\mathcal{C}$ is some additive category, let $\Theta$ be a set of objects in $\mathcal{C}$ and let $X$ be an object in $\mathcal{C}$. If the kernels $\KER{f}$ of all morphisms $f\in\Hom{\mathcal{C}}{X}{T}$ with $T\in \Theta$ exist, then the \emph{reject of $X$ in $\Theta$}, $\Rej{X}{\Theta}$, is the pullback of the monics
\[(\KER{f}:\Ker{f}\longrightarrow X)_{f\in\Hom{\mathcal{C}}{X}{T}, T \in \Theta},\]
in case this limit exists. So technically, a reject of $X$ consists of an object $\Rej{X}{\Theta}$ together with a monic $\iota=\iota_X: \Rej{X}{\Theta}\longrightarrow X$. In case the cokernel of $\iota$ exists, we denote it by $\pi:X \longrightarrow X/ \Rej{X}{\Theta}$. Rejects are formed using limits, so they are unique up to unique isomorphism. 

Noting that the intersection of submodules of a certain module is the pullback of the corresponding inclusion maps, one concludes that the notion of reject in $\Mod{\Lambda}$ mentioned in Section \ref{sec:first} is an instance of the categorial definition just presented. Since every object $X$ in our underlying abelian category $\mathcal{A}$ is artinian, then the pullback of the monomorphisms $(\KER{f})_{f\in\Hom{\mathcal{C}}{X}{T}, T \in \Theta}$ exists because it can be reduced to a pullback of finitely many of these morphisms. This means that the reject $\iota_X: \Rej{X}{\Theta}\longrightarrow X$ always exists in $\mathcal{A}$ and it is the kernel of some morphism from $X$ to a (finitary) direct sum objects in $\Theta$ (so $X/\Rej{X}{\Theta}$ is cogenerated by $\Theta$). From the universal property of kernels and pullbacks, it follows that any morphism $f:X \longrightarrow Y$ in $\mathcal{A}$ gives rise to a morphism $\Rej{f}{\Theta}:\Rej{X}{\Theta} \longrightarrow \Rej{Y}{\Theta}$ and this correspondence is functorial. Indeed, there is a monic natural tranformation $\iota: \Rej{-}{\Theta} \longrightarrow 1_{\mathcal{A}}$ and this implies that $\Rej{-}{\Theta}$ is an additive functor. Consider now a map $f:X \longrightarrow \bigoplus_{i=1}^t T_i$, with $T_i$ in $\Theta$. A standard argument shows that the reject $\iota_X$ factors through the kernel of $f$, hence $f$ factors through $\pi:X \longrightarrow X/\Rej{X}{\Theta}$. Putting together the observations made in this paragraph, we deduce that $X/\Rej{X}{\Theta}$ is the largest factor of $X$ cogenerated by $\Theta$.

We now fix a length function $\lambda:\mathcal{A}\longrightarrow [0,+\infty [$ on $\mathcal{A}$. Given $I \in \mathcal{P}_f([0,+\infty[)$, let $\mathcal{A}(< I)$ be the set of all $Y$ in $\ind{\mathcal{A}}$ such that $\mu_{\lambda}(Y) < I$. Given $X$ in $\mathcal{A}$, we construct recursively a sequence of objects $(X_i)_{i\in \N}$. As in Section \ref{sec:first}, we set $X_1=X$ and define
\begin{equation}
\label{eq:seq}
X_{i+1}=X_i/\Rej{X_i}{\mathcal{A}(<  \mu_{\lambda} (X_i))}.
\end{equation}
\begin{lem}
\label{lem:terminate}
If $X_i\neq 0$, then $\mu_{\lambda} (X_{i+1})<\mu_{\lambda} (X_i)$ and $X_{i+1}$ is a proper factor object of $X_i$. In particular, there exists $\ell \in \N_0$ such that $X_i=0$ for all $i > \ell$.
\end{lem}
\begin{proof}
Suppose that $X_i \neq 0$ for some $i \in \N$, so $\mu_{\lambda} (X_i)\neq \emptyset$. By the properties rejects, $X_{i+1}$ is the largest factor object of $X_i$ cogenerated by objects in $\mathcal{A}(< \mu_{\lambda} (X_i))$. Part (3) of Theorem \ref{thm:easy} implies that $\mu_{\lambda} (X_{i+1})< \mu_{\lambda} (X_i)$, so $X_{i+1}$ is a proper factor of $X_i$. The second statement of the lemma follows from the fact that $X$ is noetherian.
\end{proof}
Let $\ell$ be the smallest integer satisfying $X_{\ell+1}=0$. Define $\mathcal{C}_i=\Add{\bigoplus_{j\geq i} X_j}$. We claim that 
\[0=\mathcal{C}_{\ell +1}\subseteq \mathcal{C}_{\ell}  \subseteq \cdots \subseteq \mathcal{C}_{1}\]
is a left complete total rejective chain of subcategories of $\mathcal{A}$.
\begin{rem}
\label{rem:expla}
Let $Y$ be in $\mathcal{A}$. Note that $Y'=Y/ \Rej{Y}{\mathcal{A}(<  I)}$ is the largest factor object of $Y$ cogenerated by indecomposables having Gabriel--Roiter measure smaller than $I$. Part (1) of Theorem \ref{thm:easy} implies that $Y'$ is the largest factor of $Y$ with Gabriel--Roiter measure smaller than $I$. If $\mu_{\lambda}(Y)<I$, then $Y'$ is isomorphic to $Y$. However, if $\mu_{\lambda}(Y)\geq I$, then $Y'$ is a proper factor object of $Y$. By the way $\mathcal{C}_1$ was constructed and by the fact that rejects preserve finitary direct sums, we conclude that $\mu_{\lambda} (X_1), \mu_{\lambda} (X_2),\ldots ,\mu_{\lambda} (X_\ell)$ are all the possible Gabriel--Roiter measures for a nonzero object in $\mathcal{C}_1$.
\end{rem}
\begin{prop}
\label{prop:grchain}
The chain
\[0=\mathcal{C}_{\ell +1}\subset \mathcal{C}_{\ell}  \subset \cdots \subset \mathcal{C}_{1}\]
defined above is a left complete total rejective chain of subcategories of $\mathcal{A}$.
\end{prop}
\begin{proof}
Note that $\max\{\mu_{\lambda}  (X) \mid X \in \mathcal{C}_i\}=\mu_{\lambda} (X_i)$. Using that $\mu_{\lambda}  (X_{i+1})<\mu_{\lambda} (X_i)$, we deduce that $\mathcal{C}_{i+1}$ is a proper subcategory of $\mathcal{C}_i$. 

We show first that $\mathcal{C}_i$ is a left rejective subcategory of $\mathcal{C}_1$. For this, let $Y$ be an object in $\mathcal{C}_1$ and consider $Y'=Y/\Rej{Y}{\mathcal{A}(<  \mu_{\lambda}  (X_{i-1}))}$. Because rejects preserve finitary direct sums, it is not difficult to infer that $Y'$ must lie in $\mathcal{C}_i$ (see also Remark \ref{rem:expla}). Every object in $\mathcal{C}_i$ has Gabriel--Roiter measure smaller than $\mu_{\lambda} (X_{i-1})$. Since $Y'$ is the largest factor object of $Y$ whose Gabriel--Roiter measure is smaller than $\mu_{\lambda} (X_{i-1})$, then every morphism from $Y$ to an object in $\mathcal{C}_i$ factors through the canonical epic $\pi: Y \longrightarrow Y'$. In other words, $\pi$ is a left $\mathcal{C}_i$-approximation of $Y$. 

It remains to show that every nonisomorphism between indecomposable objects in $\mathcal{C}_i$ factors through an object in $\mathcal{C}_{i+1}$. For this, consider a nonisomorphism $f:Y\longrightarrow Z$ between indecomposable objects in $\mathcal{C}_i$. We show that $f$ factors through an object in $\mathcal{C}_{i+1}$. Let us assume without loss of generality that both $Y$ and $Z$ lie in $\mathcal{C}_{i}$ but not in $\mathcal{C}_{i+1}$. It follows that $\mu_{\lambda} (Y)=\mu_{\lambda} (Z)=\mu_{\lambda}  (X_i)$, so $\lambda(Y)=\lambda(Z)$ according to part (2) of Theorem \ref{thm:easy}. Part (1) of Theorem \ref{thm:easy} implies that $\mu_{\lambda} (\Ima{f})\leq \mu_{\lambda} (Z)$. We prove that $\mu_{\lambda}(\Ima{f}) < \mu_{\lambda}(Z)$. 

Suppose, by contradiction, that $\mu_{\lambda}(\Ima{f})=\mu_{\lambda}(Z)$. There exists an indecomposable summand $W$ of $\Ima{f}$ satisfying $\mu_{\lambda}(W)=\mu_{\lambda}(Z)$. Part (2) of Theorem \ref{thm:easy} yields $\lambda(W)=\lambda(Z)$. By definition of length function, the embedding of $W$ in $\Ima{f}$ postcomposed with the monic $\IMA{f}:\Ima{f}\longrightarrow Z$ must then be an isomorphism. Consequently, $f:Y \longrightarrow Z$ is an an epic satisfying $\lambda(Y)=\lambda(Z)$, hence $f$ is an isomorphism -- a contradiction. 

Therefore $\mu_{\lambda}(\Ima{f}) < \mu_{\lambda}(Z)=\mu_{\lambda}(X_i)$ and $f$ must factor through $Y'$ which is an object in $\mathcal{C}_{i+1}$, as explained in the first part of the proof.
\end{proof}

Our main result is a corollary of Proposition \ref{prop:grchain} and it refines Theorem \ref{thm:iyamat}.
\begin{cor}
\label{cor:finrepdim}
Let $X$ be an arbitrary object in an abelian length category $\mathcal{A}$ and let $\lambda$ be a length function on $\mathcal{A}$. Consider the sequence $(X_i)_{i\in \N}$ of factor objects of $X=X_1$ defined in \eqref{eq:seq} using the Gabriel--Roiter $\mu_{\lambda}$ of $\mathcal{A}$ associated to $\lambda$. Let $\ell\in \mathbb{N}_0$ be the number of nonzero objects of $(X_i)_{i\in \N}$ and set $\overline{X}=\bigoplus_{i=1}^{\ell}X_i$. Then, the ring $\Gamma=\End{\mathcal{A}}{\overline{X}}$ has finite global dimension. More precisely, $\Gamma$ is a left strongly quasihereditary ring and its global dimension is at most $\ell$. Moreover, if $Y$ is an object in $\Add{\overline{X}}$, then it satisfies $\mu_{\lambda}(Y)=\mu_{\lambda}(X_i)$ for exactly one $i$ ($1\leq i\leq \ell$) and the projective dimension of the top of $\Hom{\mathcal{A}}{Y}{\overline{X}}$ is at most $i$.
\end{cor}
\begin{proof}
Note that the category $\mathcal{C}_1$ in the statement of Proposition \ref{prop:grchain} has finite type. The corollary then follows from Proposition \ref{prop:grchain}, Theorem \ref{thm:reject} and Remark \ref{rem:expla}.
\end{proof}
Dlab and Ringel's result in \cite{MR943793} and Iyama's theorem in \cite{iyama} indicate that quasihereditary algebras are quite prevalent in representation theory, and Corollary \ref{cor:finrepdim} provides additional evidence for the ubiquity of quasihereditary rings.

The advantage of the procedure described in Corollary \ref{cor:finrepdim} in relation to Iyama's construction (recall Example \ref{ex:iyama}) is that, by varying the length functions on $\mathcal{A}=\Mod{\Lambda}$ (i.e.~by attributing different `weights' to the simple modules), one is usually able to generate various complements $X'$ of a fixed finitely generated $\Lambda$-module and several algebras $\Gamma=\End{\Lambda}{X\oplus X'}$ of finite global dimension. The next examples illustrate how the method described in Corollary \ref{cor:finrepdim} can be used in pratice.
\begin{ex}
\label{ex:last}
Consider the bound quiver algebra $\Lambda=k Q/I$ in Example \ref{ex:easy} and, once again, denote the projective indecomposable $P_1$ by $X_1=X$. Let $\lambda$ be the length function on $\Mod{\Lambda}$ mapping $L_1$ to 1 and $L_2$ to 2. Then $\mu_{\lambda}(X_1)=\{1,2,5\}$ and, as in Example \ref{ex:easy},
\[
X_2=X_1/\Rej{X_1}{\mathcal{A}(<  \mu_{\lambda} (X_1))}=\begin{tikzcd}[ampersand replacement=\&, row sep=tiny, column sep=tiny,]
\& 1 \arrow[dash]{dr} \arrow[dash]{dl} \& \\
1  \& \& 2 
\end{tikzcd}.
\]
Note that $\mu_{\lambda} (X_2)=\{1,4\}$, 
\[
X_3=X_2/\Rej{X_2}{\mathcal{A}(<  \mu_{\lambda} (X_2))}=\begin{tikzcd}[ampersand replacement=\&, row sep=tiny, column sep=tiny]
 1 \arrow[dash]{d}  \\
2 
\end{tikzcd},
\]
$\mu_{\lambda}(X_3)=\{2,3\}$ and $X_4=0$. By Corollary \ref{cor:finrepdim}, the corresponding algebra $\Gamma=\End{\Lambda}{X_1 \oplus X_2 \oplus X_3}$ has global dimension at most $3$ and the top of $\End{\Lambda}{X_2,X_1\oplus X_2\oplus X_3}$ has projective dimension at most $2$.
\end{ex}
\begin{ex}
Consider the same algebra $\Lambda=k Q/I$ and the same module $X_1=P_1$ as before, but now let $\lambda'$ be the length function which sends $L_1$ to 2 and $L_2$ to 1. Then
\[
X_2=X_1/\Rej{X_1}{\mathcal{A}(<  \mu_{\lambda'} (X_1))}=\begin{tikzcd}[ampersand replacement=\&, row sep=tiny, column sep=tiny,]
1 \arrow[dash]{d}\\
1  \arrow[dash]{d}\\
1
\end{tikzcd}, \,
X_3=X_2/\Rej{X_1}{\mathcal{A}(<  \mu_{\lambda'} (X_2))}=\begin{tikzcd}[ampersand replacement=\&, row sep=tiny, column sep=tiny,]
1 \arrow[dash]{d}\\
1
\end{tikzcd},
\]
$X_4=L_1$ and $X_5=0$. The corresponding algebra $\Gamma$ differs from the ones in Examples \ref{ex:last} and \ref{ex:easy}.
\end{ex}
\begin{ex}
\label{ex:verylast}
Let $\Lambda$ be the bound quiver algebra $k Q/I$ where $I=\langle \alpha^3, \beta\alpha, \gamma\alpha\rangle$ and
\[Q=
\begin{tikzcd}[ampersand replacement=\&]
  \overset{1}{\circ} \arrow[loop left]{}{\alpha} \arrow[bend left]{r}{\beta}
  \arrow[bend right]{r}[swap]{\gamma}\&
    \overset{2}{\circ}
\end{tikzcd}.
\]
Let $X_1=X=\Lambda \oplus D(\Lambda)=P_1\oplus P_2 \oplus Q_1 \oplus Q_2$ be the `smallest' generator-cogenerator of $\Mod{\Lambda}$ (here $Q_i$ denotes the injective $\Lambda$-module with socle $L_i$). Let $\lambda$ be the usual length function on $\Mod{\Lambda}$, i.e.~the one sending both $L_1$ and $L_2$ to 1. By applying Corollary \ref{cor:finrepdim} to this data, we obtain $\ell=5$ and the basic version of the corresponding module $\overline{X}$ is isomorphic to
\[P_1\oplus P_2 \oplus Q_1 \oplus Q_2 \oplus (Q_1/L_1) \oplus L_1 \oplus (P_1/L_1).\]
\end{ex}
\begin{ex}
Let $\Lambda$ and $X$ be as in Example \ref{ex:verylast}, but now consider the length function $\lambda'$ mapping $L_1$ to 1 and $L_2$ to 2. In this case, Corollary \ref{cor:finrepdim} yields $\ell=8$ and the basic version of the corresponding module $\overline{X}$ is isomorphic to
\[P_1\oplus P_2 \oplus Q_1 \oplus Q_2 \oplus (Q_1/L_1) \oplus L_1 \oplus (P_1/L_1) \oplus ((P_1/L_1)/L_1).\]
\end{ex}
\begin{ex}
Again, let $\Lambda$ and $X$ be as in Example \ref{ex:verylast}, and consider the length function $\lambda''$ sending $L_1$ to 2 and $L_2$ to 1. By  applying Corollary \ref{cor:finrepdim} to this data, we get $\ell=6$ and the basic version of the corresponding module $\overline{X}$ is isomorphic to 
\[P_1\oplus P_2 \oplus Q_1 \oplus Q_2 \oplus (Q_1/L_1) \oplus L_1.\]
\end{ex}
\begin{rem}
There exists a notion dual to that of Gabriel--Roiter measure, called the Gabriel--Roiter comeasure (see \cite[Appendix C]{RINGEL2005726}). By passing the results in this section to the opposite category, we obtain statements about right complete total rejective chains constructed using the Gabriel--Roiter comeasure.
\end{rem}
\bibliographystyle{plain}
\bibliography{QHAlg}

\begin{thebibliography}{10}

\bibitem{MR1245487}
F.~W. Anderson and K.~R. Fuller.
\newblock {\em Rings and categories of modules}, volume~13 of {\em Graduate
  Texts in Mathematics}.
\newblock Springer-Verlag, New York, second edition, 1992.

\bibitem{MR0074406}
M.~Auslander.
\newblock On the dimension of modules and algebras. {III}. {G}lobal dimension.
\newblock {\em Nagoya Math. J.}, 9:67--77, 1955.

\bibitem{quenmary}
M.~Auslander.
\newblock {R}epresentation dimension of {A}rtin algebras.
\newblock Lecture Notes, Queen Mary College, London, 1971.

\bibitem{MR0349747}
M.~Auslander.
\newblock Representation theory of {A}rtin algebras. {I}, {II}.
\newblock {\em Comm. Algebra}, 1:177--268; ibid. 1 (1974), 269--310, 1974.

\bibitem{MR0424874}
M.~Auslander.
\newblock Large modules over {A}rtin algebras.
\newblock In {\em Algebra, topology, and category theory (a collection of
  papers in honor of {S}amuel {E}ilenberg)}, pages 1--17. Academic Press, New
  York, 1976.

\bibitem{ausrei}
M.~Auslander and I.~Reiten.
\newblock Applications of contravariantly finite subcategories.
\newblock {\em Adv. Math.}, 86(1):111--152, 1991.

\bibitem{MR1476671}
M.~Auslander, I.~Reiten, and S.~O. Smal{\o}.
\newblock {\em Representation theory of {A}rtin algebras}, volume~36 of {\em
  Cambridge Studies in Advanced Mathematics}.
\newblock Cambridge University Press, Cambridge, 1997.
\newblock Corrected reprint of the 1995 original.

\bibitem{2017arXiv170608301C}
H.~{Chen}, M.~{Fang}, O.~{Kerner}, S.~{Koenig}, and K.~{Yamagata}.
\newblock {Rigidity dimension - a homological dimension measuring resolutions
  of algebras by algebras of finite global dimension}.
\newblock {\em arXiv e-prints}, June 2017.

\bibitem{MR3510398}
T.~Conde.
\newblock The quasihereditary structure of the {A}uslander--{D}lab--{R}ingel
  algebra.
\newblock {\em J. Algebra}, 460:181--202, 2016.

\bibitem{MR943793}
V.~Dlab and C.~M. Ringel.
\newblock Every semiprimary ring is the endomorphism ring of a projective
  module over a quasihereditary ring.
\newblock {\em Proc. Amer. Math. Soc.}, 107(1):1--5, 1989.

\bibitem{yey}
V.~Dlab and C.~M. Ringel.
\newblock Filtrations of right ideals related to projectivity of left ideals.
\newblock In {\em S{\'e}minaire d'Alg{\`e}bre Dubreil-Malliavin}, volume 1404
  of {\em Lecture Notes in Math.}, pages 95--107. Springer, 1989.

\bibitem{MR987824}
V.~Dlab and C.~M. Ringel.
\newblock Quasi-hereditary algebras.
\newblock {\em Illinois J. Math.}, 33(2):280--291, 1989.

\bibitem{MR1211481}
V.~Dlab and C.~M. Ringel.
\newblock The module theoretical approach to quasi-hereditary algebras.
\newblock In {\em Representations of algebras and related topics ({K}yoto,
  1990)}, volume 168 of {\em London Math. Soc. Lecture Note Ser.}, pages
  200--224. Cambridge Univ. Press, Cambridge, 1992.

\bibitem{MR0340377}
P.~Gabriel.
\newblock Indecomposable representations. {II}.
\newblock In {\em Symposia {M}athematica, {V}ol. {XI} ({C}onvegno di {A}lgebra
  {C}ommutativa, {INDAM}, {R}ome, 1971)}, pages 81--104. Academic Press,
  London, 1973.

\bibitem{KacMoodyAdv}
C.~Gei{\ss}, B.~Leclerc, and J.~Schr{\"o}er.
\newblock Kac-{M}oody groups and cluster algebras.
\newblock {\em Adv. Math.}, 228(1):329--433, 2011.

\bibitem{MR3373366}
B.~Grimeland and K.~M. Jacobsen.
\newblock Abelian quotients of triangulated categories.
\newblock {\em J. Algebra}, 439:110--133, 2015.

\bibitem{MR0286859}
M.~Harada and Y.~Sai.
\newblock On categories of indecomposable modules. {I}.
\newblock {\em Osaka J. Math.}, 7:323--344, 1970.

\bibitem{iyama}
O.~Iyama.
\newblock Finiteness of representation dimension.
\newblock {\em Proc. Amer. Math. Soc.}, 131:1011--1014, 2002.

\bibitem{MR1953711}
O.~Iyama.
\newblock A proof of {S}olomon's second conjecture on local zeta functions of
  orders.
\newblock {\em J. Algebra}, 259(1):119--126, 2003.

\bibitem{2003math.....11281I}
O.~{Iyama}.
\newblock {Rejective subcategories of artin algebras and orders}.
\newblock {\em arXiv e-prints}, November 2003.

\bibitem{IyamaReiten}
O.~Iyama and I.~Reiten.
\newblock 2-{A}uslander algebras associated with reduced words in {C}oxeter
  groups.
\newblock {\em Int. Math. Res. Not. IMRN}, (8):1782--1803, 2011.

\bibitem{MR2346934}
H.~Krause.
\newblock An axiomatic characterization of the {G}abriel-{R}oiter measure.
\newblock {\em Bull. Lond. Math. Soc.}, 39(4):550--558, 2007.

\bibitem{MR3431480}
H.~Krause.
\newblock Krull-{S}chmidt categories and projective covers.
\newblock {\em Expo. Math.}, 33(4):535--549, 2015.

\bibitem{xi}
Y.~Lin and C.~Xi.
\newblock Semilocal modules and quasi-hereditary algebras.
\newblock {\em Arch. Math.}, 60(6):512--516, 1993.

\bibitem{Moosonee}
B.~J. Parshall and L.~L. Scott.
\newblock Derived categories, quasi-hereditary algebras, and algebraic groups.
\newblock In {\em Proc. Ottawa--Moosonee Workshop}, volume~3, pages 1--105.
  Carleton--Ottawa Math. Lecture Note Series, 1988.

\bibitem{RINGEL2005726}
C.~M. Ringel.
\newblock The {G}abriel-{R}oiter measure.
\newblock {\em Bull. Sci. Math.}, 129(9):726--748, 2005.

\bibitem{grmfringel}
C.~M. Ringel.
\newblock Foundation of the representation theory of {A}rtin algebras, {U}sing
  the {G}abriel-{R}oiter measure.
\newblock In {\em Trends in Representation Theory of Algebras and Related
  Topics (Workshop Queretaro, Mexico, 2004)}, volume 406 of {\em Contemporary
  Math.}, pages 105--135. Amer. Math. Soc., Providence, RI, 2006.

\bibitem{RingelIyama}
C.~M. Ringel.
\newblock Iyama's finiteness theorem via strongly quasi-hereditary algebras.
\newblock {\em J. Pure Appl. Algebra}, 214(9):1687---1692, 2010.

\bibitem{MR0379578}
C.~M. Ringel and H.~Tachikawa.
\newblock {${\rm QF}-3$} rings.
\newblock {\em J. Reine Angew. Math.}, 272:49--72, 1974.

\bibitem{MR0238893}
A.~V. Ro\u{\i}ter.
\newblock Unboundedness of the dimensions of the indecomposable representations
  of an algebra which has infinitely many indecomposable representations.
\newblock {\em Izv. Akad. Nauk SSSR Ser. Mat.}, 32:1275--1282, 1968.

\bibitem{MR2231960}
R.~Rouquier.
\newblock Representation dimension of exterior algebras.
\newblock {\em Invent. Math.}, 165(2):357--367, 2006.

\bibitem{MR0349740}
H.~Tachikawa.
\newblock {\em Quasi-{F}robenius rings and generalizations. {${\rm QF}-3$} and
  {${\rm QF}-1$} rings}.
\newblock Lecture Notes in Mathematics, Vol. 351. Springer-Verlag, Berlin-New
  York, 1973.
\newblock Notes by C. M. Ringel.

\bibitem{Mayu}
M.~Tsukamoto.
\newblock Strongly quasi-hereditary algebras and rejective subcategories.
\newblock {\em Nagoya Math. J.}, pages 1--29, 2018.

\end{thebibliography}

\end{document}